\documentclass[11pt]{amsart}
\usepackage{amsmath,amssymb,latexsym,dsfont,mathrsfs}
\usepackage[small]{caption}
\usepackage{graphicx,color,mathrsfs,tikz}%wasysym,overpic,tikz,
\usepackage{subfigure,color}
\usepackage{cite}
\usepackage[colorlinks=true,urlcolor=blue,
citecolor=red,linkcolor=blue,linktocpage,pdfpagelabels,
bookmarksnumbered,bookmarksopen]{hyperref}
\usepackage[italian,english]{babel}
\usepackage{units}
\usepackage{enumitem}
\usepackage[left=2.65cm,right=2.65cm,top=3cm,bottom=3cm]{geometry}
\usepackage[hyperpageref]{backref}

\usepackage[colorinlistoftodos,prependcaption]{todonotes}
\numberwithin{equation}{section}
\newtheorem{theorem}{Theorem}[section]
\newtheorem{proposition}{Proposition}[section]
\newtheorem{lemma}{Lemma}[section]
\newtheorem{remark}{Remark}[section]

\newtheorem{corollary}{Corollary}[section]

\theoremstyle{definition}

\renewcommand{\epsilon}{\eps}
\renewcommand{\i}{{\rm i}}

\newcommand{\C}{{\mathscr A}}

\newcommand{\R}{{\mathbb R}}
\newcommand{\mS}{{\mathbb S}}

\newcommand{\Z}{{\mathbb Z}}

\newcommand{\eps}{\varepsilon}

\newcommand{\pnorm}[2][]{\if #1'' \left|#2\right|_p \else \left|#2\right|_{#1} \fi}

\newcommand{\loc}{{\rm loc}}

\newcommand{\rr}{\tau}

\renewcommand{\theta}{\vartheta}
\def\Xint#1{\mathchoice
	{\XXint\displaystyle\textstyle{#1}}%
	{\XXint\textstyle\scriptstyle{#1}}%
	{\XXint\scriptstyle\scriptscriptstyle{#1}}%
	{\XXint\scriptscriptstyle\scriptscriptstyle{#1}}%
	\!\int}
\def\XXint#1#2#3{{\setbox0=\hbox{$#1{#2#3}{\int}$ }
		\vcenter{\hbox{$#2#3$ }}\kern-.6\wd0}}

\def\dashint{\Xint-}

\DeclareMathOperator{\supp}{supp}
\title{On Hardy and Caffarelli-Kohn-Nirenberg inequalities}

\author[H.-M. Nguyen]{Hoai-Minh Nguyen}
\author[M.\ Squassina]{Marco Squassina}

\address[H.-M. Nguyen]{Department of Mathematics \newline\indent
	EPFL SB CAMA \newline\indent
	Station 8 CH-1015 Lausanne, Switzerland}
\email{hoai-minh.nguyen@epfl.ch}

\address[M.\ Squassina]{Dipartimento di Matematica e Fisica \newline\indent
	Universit\`a Cattolica del Sacro Cuore \newline\indent
	Via dei Musei 41, I-25121 Brescia, Italy}
\email{marco.squassina@unicatt.it}

\thanks{The second author is member of {\em Gruppo Nazionale per l'Analisi Ma\-te\-ma\-ti\-ca, la Probabilit\`a e le loro Applicazioni} (GNAMPA) of the {\em Istituto Nazionale di Alta Matematica} (INdAM). Part  of  this  paper  was  written  during  a  visit  of  M.S.  in
Lausanne in June 2017.  The hosting institution is gratefully acknowledged.}



\begin{document}

\begin{abstract}
We establish improved versions of the 
Hardy and Caffarelli-Kohn-Nirenberg inequalities by replacing the standard
Dirichlet energy with some nonlocal nonconvex functionals which have
been involved in estimates for the topological degree of  continuous maps from a sphere into itself and characterizations of Sobolev spaces. 
\end{abstract}
\maketitle

\begin{center}
	\begin{minipage}{9.4cm}
		\small
		\tableofcontents
	\end{minipage}
\end{center}

\medskip

\section{Introduction}
%\subsection{Overview}
In many branches of mathematical physics, harmonic and stochastic analysis, the 
classical {\em Hardy inequality} plays a central role. It states that, if $1\leq p<d$,
\begin{equation*}
%\label{ineq-Hardy}
\left(\frac{d-p}{p}\right)^p\int_{\R^d}\frac{|u|^p}{|x|^p} dx\leq \int_{\R^d}|\nabla u|^p dx,
\end{equation*}
for every $u\in C^1_c(\R^d)$ with optimal constant which, contrary to the Sobolev inequality, is never attained.  
 Another class of relevant inequalities is given by the so called {\em Caffarelli-Kohn-Nirenberg inequalities} \cite{CKN,CKN-original}.  Precisely, let $p \ge 1$, $q \ge 1$, $\rr>0$, $0 <   a \le   1$, $\alpha, \, \beta, \, \gamma \in \R$ be such that 
%\begin{equation}\label{CKN-1}
%\frac{1}{\rr} + \frac{\gamma}{d} \ge 0,  
%\end{equation}
\begin{equation}\label{totototo}
%\label{CKN-a-1}
\frac{1}{\rr} + \frac{\gamma}{d}, \quad  \frac{1}{p} + \frac{\alpha}{d}, \quad  \frac{1}{q} + \frac{\beta}{d}  > 0,
\end{equation}
\begin{equation*}
\frac{1}{\rr} + \frac{\gamma}{d} = a \left(\frac{1}{p} + \frac{\alpha - 1}{d} \right) + (1-a) \left( \frac{1}{q} + \frac{\beta}{d} \right), 
\end{equation*}
and, with $\gamma = a \sigma + (1 -a) \beta$, 
\begin{equation*}
%\label{CKN-a-2}
0 \le  \alpha - \sigma
\end{equation*}
and 
\begin{equation*}
%\label{CKN-a-3}
\alpha - \sigma \le  1 \quad  \mbox{ if }  \quad \frac{1}{\rr} + \frac{\gamma}{d}  = \frac{1}{p} + \frac{\alpha - 1}{d}. 
\end{equation*}
%\begin{equation*}
%0 \le  \alpha - \sigma \le 1.   
%\end{equation*}
Then, for every $u \in C^1_{c}(\R^d)$,
\begin{equation*}
%\label{ineq-CKN}
\| |x|^\gamma u   \|_{L^\rr(\R^d)} \le C \| |x|^\alpha \nabla u \|_{L^p(\R^d)}^{a} \| |x|^\beta u   \|_{L^q(\R^d)}^{(1-a)},
\end{equation*}
for some positive constant $C$ independent of $u$.  This inequality has been an object of a large amount of improvement and extensions to more general frameworks. 

%\textcolor{red}{Specify Hardy's inequality $\alpha = 1$ and $\rr = p$}. 

\medskip 
In the non-local case, 
it was shown in \cite{FS, MS} that  there exists $C>0$, independent of $0<\delta<1$, such that 
\begin{equation}
\label{sobhardypfin}
C\int_{\R^d}\frac{|u(x)|^{p}}{|x|^{p\delta}}dx
\leq J_\delta(u),
%(1-\delta)\int_{\mathbb{R}^{2d}}\frac{|u(x)-u(y)|^p}{|x-y|^{d+p\delta}}ddy,
\end{equation}
for all $u\in C^1_c(\R^d)$, where 
$$
J_\delta(u):=(1-\delta)\iint_{\mathbb{R}^{2d}}\frac{|u(x)-u(y)|^p}{|x-y|^{d+p\delta}}dxdy.
$$ 
In light of the results of Bourgain, Brezis, and Mironescu \cite{bourg,bourg2} and an refinement of Davila \cite{Davila}, it holds
\begin{equation*}
%\label{BBM}
\lim_{\delta \searrow 0} J_\delta(u)= K_{d, p}  \int_{\R^d} |\nabla u|^pdx,\quad  \mbox{ for } u\in W^{1, p}(\R^d),  \qquad
K_{d, p} : = \frac{1}{p}\int_{\mS^{d-1}} |\boldsymbol{e} \cdot \sigma|^p \, d \sigma,
\end{equation*}
 for some $\boldsymbol{e} \in \mS^{d-1}$, being $\mS^{d-1}$ the unit sphere in $\R^d$. 
 This
 allows to recover the classical Hardy inequality from  \eqref{sobhardypfin} by letting
 $\delta \searrow 0$. Various problems related to $J_\delta$ are considered in \cite{bre, BHN, BHN2, BHN3, ponce, PS}.  The full range of Caffarelli-Kohn-Nirenberg inequalities and their variants were established in \cite{NgS3} (see  \cite{f-CKN} for partial results in the case $a=1$). 

\medskip 
%\subsection{Main results}

Set, for $p \ge 1$, $\Omega$ a measurable set of $\R^d$, and $u \in L^1_{\loc}(\Omega)$,  
\begin{equation*}
I_{\delta}(u, \Omega) := \mathop{\int_{\Omega} \int_{\Omega}}_{\{|u(x) - u(y)| > \delta \}} \frac{\delta^p}{|x - y|^{d + p}} \, dx \, dy.
\end{equation*}
In the case, $\Omega = \R^d$, we simply denote $I_\delta(u, \R^d)$ by $I_{\delta} (u)$. The quantity $I_\delta$  with $p=d$ has its roots in estimates for the topological degree of continuous maps from a sphere into itself  in \cite{BBNg1, Ngdegree1}. This also appears in characterizations of Sobolev spaces \cite{BourNg, nguyen06, NgSob2, BHN, BHN-0} and related contexts \cite{BrezisNguyen,BHN, BHN-0,  NgGammaCRAS, NgGamma,  Ng11, NgS1, NgS2}.  It is known that (see \cite[Theorem 2]{nguyen06} and \cite[Proposition 1]{BHN}), for $p \ge 1$, 
\begin{equation}\label{Idelta-p1}
\lim_{\delta \searrow 0}I_\delta (u) = 
K_{d, p} \int_{\R^d} |\nabla u|^p \, dx,\quad \mbox{ for } u \in  C^1_{c}(\R^d) \; \footnote{In the case $p>1$, one can take $u \in W^{1, p}(\R^d)$ in \eqref{Idelta-p1}. Nevertheless, \eqref{Idelta-p1} does not hold for $u \in W^{1, 1}(\R^d)$ in the case $p=1$. An example for this is due to Ponce presented in  \cite{nguyen06}.}
\end{equation}
and, for $p >  1$,  
\begin{equation}\label{Idelta-p2}
I_\delta (u) \le C_{d, p} \int_{\R^d} |\nabla u|^p \, dx, \quad
\mbox{ for } u \in W^{1, p} (\R^d), 
\end{equation}
for some positive constant $C_{d, p}$ independent of $u$.

\medskip 
The aim of this paper is to get improved versions of the local Hardy and Caffarelli-Kohn-Nirenberg type inequalities and their variants which involve nonlinear nonlocal nonconvex energies $I_\delta(u)$ and its related quantities. In what follows for $R> 0$, $B_R$ denotes the open ball of $\R^d$ centered at the origin of radius $r$.   Our first main result concerning Hardy's inequality is:

\begin{theorem}[Improved Hardy inequality]
	\label{thm-Hardy} Let $d \ge 1$, $p \ge 1$, $0< r < R$,  and $u \in L^p(\R^d)$. We have
\begin{enumerate}
\item[i)] if $1 \le p < d$ and $\supp u \subset B_R$,   then 
\begin{equation*}
\int_{\R^d} \frac{|u(x)|^p}{|x|^p} \, dx \le C \left(I_\delta(u) + R^{d-p} \delta^p  \right), 
\end{equation*}

\item[ii)] if $p > d$ and $\supp u \subset \R^d \setminus B_r$, then
\begin{equation*}
\int_{\R^d} \frac{|u(x)|^p}{|x|^p} \, dx \le C \left(I_\delta(u) + r^{d-p} \delta^p  \right), 
\end{equation*}

\item[iii)] if $p = d \ge 2$ and $\supp u \subset B_R$, then
\begin{equation*}
\int_{\R^d \setminus B_r} \frac{|u(x)|^d}{|x|^d \ln^d (2R/ |x|)} \, dx \le C \left(I_\delta(u) +   \ln (2R/ r) \delta^d  \right), 
\end{equation*}
%and, for $\beta > 3$,  
%\begin{equation*}
%\int_{\R^d} \frac{|u(x)|^d}{|x|^d \ln^{\beta} (2R/ |x|)} \, dx \le C_\beta \left(I_\delta(u) +   \delta^d  \right).   
%\end{equation*}

\item[iv)] if $p = d \ge 2$ and $\supp u \subset \R^d \setminus B_r$, then
\begin{equation*}
\int_{B_R} \frac{|u(x)|^d}{|x|^d \ln^d (2|x|/r)} \, dx \le C \left(I_\delta(u) +   \ln (2R/ r) \delta^d  \right),  
\end{equation*}
 
\end{enumerate}
where  $C$  denotes a positive constant depending only on $p$ and $d$. 
 
%Here  $C$ (resp. $C_\beta$)  denotes a positive constant depending only on $p$ and $d$ (resp.  $p$, $d$ and $\beta$). 

\end{theorem}

In light of \eqref{Idelta-p1}, by letting $\delta \to 0$, one obtains variants of $i), ii), iii), iv)$ of Theorem~\ref{thm-Hardy} where the RHS is replaced by $ C \int_{\R^d} |\nabla u|^p \, dx$; see Proposition~\ref{pro1} for a more general version. 
By \eqref{Idelta-p1} and \eqref{Idelta-p2}, Theorem~\ref{thm-Hardy} provides improvement of Hardy's inequalities in the case $p>1$. 

%Recall that, by   \eqref{Idelta-p2}, one can replace $I_\delta(u)$

%\begin{remark}
%\rm \textcolor{red}{Remark on $I_\delta$ in \eqref{Idelta-p2}.}  
%\end{remark}

%\begin{enumerate}
%\item[i)] if $1 \le p < d$ and $u \in C^1_{c}(\R^d)$,   then 
%\begin{equation*}
%\int_{\R^d} \frac{|u(x)|^p}{|x|^p} \, dx \le C \int_{\R^d} |\nabla u|^p \, dx,  
%\end{equation*}
%
%\item[ii)] if $p = d$ and $u \in C^1(\R^d)$ with $\supp u \subset B_R$, then
%\begin{equation*}
%\int_{\R^d} \frac{|u(x)|^d}{|x|^d \ln^d (2R/ |x|)} \, dx \le C \int_{\R^d} |\nabla u|^p \, dx, 
%\end{equation*}
%%and, for $\beta > 3$,  
%%\begin{equation*}
%%\int_{\R^d} \frac{|u(x)|^d}{|x|^d \ln^{\beta} (2R/ |x|)} \, dx \le C_\beta \left(I_\delta(u) +   \delta^d  \right).   
%%\end{equation*}
%
%\item[iii)] if $p = d$ and $u \in C^1(\R^d)$ with $\supp u \subset
%\R^d\setminus B_r$, then
%\begin{equation*}
%\int_{\R^d} \frac{|u(x)|^d}{|x|^d \ln^d (2|x|/r)} \, dx \le C  \int_{\R^d} |\nabla u|^p \, dx,  
%\end{equation*}
%
%\item[iv)] if $p > d$ and $u \in C^1(\R^d)$ with $\supp u \subset\R^d\setminus B_r$, then
%\begin{equation*}
%\int_{\R^d} \frac{|u(x)|^p}{|x|^p} \, dx \le C \int_{\R^d} |\nabla u|^p \, dx. 
%\end{equation*}
%
% 
%\end{enumerate}

\medskip 

We next discuss an improved version of Caffarelli-Kohn-Nirenberg in the case the exponent $a=1$. The more general case is considered in Theorem~\ref{CKN-g} (see also Proposition~\ref{CKN-g-1}).  
Set, for $p \ge 1$, $\alpha \in \R$, and $\Omega$ a measurable subset of $\R^d$,
\begin{equation*}
%\label{def-Ialpha}
I_{\delta}(u, \Omega, \alpha) := \mathop{\int_{\Omega} \int_{\Omega}}_{\{|u(x) - u(y)| > \delta \}} \frac{\delta^p |x|^{p \alpha}}{ |x - y|^{d + p}} \, dx \, dy, \quad \mbox{ for } u \in L^1_{\loc}(\Omega). 
\end{equation*}
If $\Omega = \R^d$, we simply denote $I_\delta(u, \R^d, \alpha)$ by $I_{\delta} (u, \alpha)$. 
We have

% improvement of a case of Caffarelli-Kohn-Nirenberg's inequality. 

\begin{theorem}[Improved Caffarelli-Kohn-Nirenberg's inequality for $a=1$]
	\label{thm-CKN} Let $d \ge 2$, $1 < p < d$, $\rr > 0$,  $0< r < R$, and $u \in L^p_{\loc}(\R^d)$. Assume that 
	$$
	\frac{1}{\rr} + \frac{\gamma}{d} = \frac{1}{p} + \frac{\alpha -1}{d} \quad \mbox{ and } \quad 
0 \le \alpha - \gamma \le 1. 
$$
We have 
\begin{enumerate}	
	
\item[i)] if $d-p + p \alpha > 0$ and $\supp u \subset B_R$, then 
\begin{equation*}
	\left(\int_{\R^d} |x|^{\gamma \rr} |u(x)|^\rr \, dx \right)^{p/\rr} \le C \left( I_{\delta} (u, \alpha)+ R^{d-p + p \alpha } \delta^p \right), 
\end{equation*}

\item[ii)] if $d-p + p \alpha <  0$ and $\supp u \subset \R^d \setminus B_r$, then 
\begin{equation*}
	\left(\int_{\R^d} |x|^{\gamma \rr} |u(x)|^\rr\, dx \right)^{p/\rr} \le C \left( I_{\delta} (u, \alpha) + r^{d-p + p \alpha} \delta^p \right), 
\end{equation*}

\item[iii)] if $d-p +  p \alpha = 0$, $\rr > 1$,  and $\supp u \subset B_R$, then 
\begin{equation*}
	\left(\int_{\R^d \setminus B_r} \frac{|x|^{\gamma \rr}   |u(x)|^\rr}{\ln^{\rr} (2 R/ |x|) } \, dx \right)^{p/\rr} \le C \Big( I_{\delta} (u, \alpha)+ \ln (2R/ r) \delta^p \Big), 
\end{equation*}

\item[iv)] if $d-p + p \alpha = 0$, $\rr > 1$,  and $\supp u \subset \R^d \setminus B_r$, then 
\begin{equation*}
	\left(\int_{B_R} \frac{|x|^{\gamma \rr}   |u(x)|^\rr}{\ln^{\rr} (2 |x| /r) }\, dx \right)^{p/\rr} \le C \Big( I_{\delta} (u, \alpha)+ \ln (2R/ r) \delta^p \Big).
\end{equation*}

%and,  for $ \beta > 3 $, 
%\begin{equation*}
%	\left(\int_{\R^d} \frac{|u(x)|^q}{|x|^{bp} \ln^\beta (2 R/ |x|)} \, dx \right)^{p/q} \le C_\beta \left( J_\delta(u)+  \delta^p \right). 
%\end{equation*}

\end{enumerate}	
Here  $C$  denotes a positive constant independent of $u$, $r$, and $R$. 
% (resp.  $p$, $d$ and $\beta$). 
\end{theorem}

%\textcolor{red}{Explain the exponent!}

\begin{remark} \rm In contrast with Theorem~\ref{thm-Hardy},  in Theorem~\ref{thm-CKN}, we assume that $1 < p < d$. This assumption is required due to the use of  Sobolev's inequality related to $I_\delta(u, \Omega, 0)$ (see  Lemmas~\ref{lem-S} and \ref{lem-S-P-3}). 
\end{remark}

%\begin{remark} \rm Using Fatou's lemma, as in the proof of \cite[Proposition 1]{BHN}, one has 
%$$
%\liminf_{\delta }
%$$
%\end{remark}

\begin{remark} \label{rem-Muck}
\rm Using the theory of maximal functions with weights due to  Muckenhoupt \cite{Muck} (see also \cite{CF}), one can bound
$I_\delta(u, \alpha)$  by $C \int_{\R^d} |x|^{p \alpha} |\nabla u|^p \, dx$ for $-1/ p < \alpha < 1 - 1/p$ and get an improvement of Caffarelli-Kohn-Nirenberg's inequality for $a=1$ via Theorem~\ref{thm-CKN} and for  $0 < a < 1$ and $0 \le \alpha - \sigma \le 1$ via Theorem~\ref{CKN-g} in Section~\ref{ckn}. The details of this fact are given in Remark~\ref{rem-Maximal} (see also Remark~\ref{rem-CKN-g} for a different approach covering a more general result).  
\end{remark}

We later prove a general version of Theorem~\ref{thm-CKN} in  Theorem~\ref{CKN-g}, where $0 < a \le 1$,   which implies 
Proposition~\ref{CKN-g-1} by interpolation. As a consequence of Theorem~\ref{CKN-g} (see also Remark~\ref{rem-CKN-g}) and Proposition~\ref{CKN-g-1}, we have

\begin{proposition}\label{pro1} Let $p \ge 1$, $q \ge 1$, $\rr > 0$, $0 <   a \le   1$, $\alpha, \, \beta, \, \gamma \in \R$ be such that 
%\begin{equation}\label{CKN-1}
%\frac{1}{\rr} + \frac{\gamma}{d} \ge 0,  
%\end{equation}
\begin{equation*}
\frac{1}{\rr} + \frac{\gamma}{d} = a \left(\frac{1}{p} + \frac{\alpha - 1}{d} \right) + (1-a) \left( \frac{1}{q} + \frac{\beta}{d} \right), 
\end{equation*}
and, with $\gamma = a \sigma + (1 -a) \beta$, 
\begin{equation*}
%\label{CKN-sign}
0 \le  \alpha - \sigma
\end{equation*}
and 
\begin{equation*}
%\label{CKN-a-3}
\alpha - \sigma \le  1 \quad  \mbox{ if }  \quad \frac{1}{\rr} + \frac{\gamma}{d}  = \frac{1}{p} + \frac{\alpha - 1}{d}. 
\end{equation*}
We have, for $u \in C^1_{c}(\R^d)$, 
\begin{enumerate}

\item[A1)] if $1/ \rr + \gamma/ d > 0$,  then 
\begin{equation*}
\left(\int_{\R^d}  |x|^{\gamma \rr}|u|^{\rr} \, dx \right)^{1/\rr} \le C \| |x|^\alpha \nabla u \|_{L^p(\R^d)}^{a} \| |x|^\beta u   \|_{L^q(\R^d)}^{(1-a)},
\end{equation*}

\item[A2)] if $1/ \rr + \gamma/ d <  0$ and  $\supp u \subset \R^d \setminus \{0 \}$,  then 
\begin{equation*}
\left(\int_{\R^d}  |x|^{\gamma \rr}|u|^{\rr} \, dx \right)^{1/\rr} \le C \| |x|^\alpha \nabla u \|_{L^p(\R^d)}^{a} \| |x|^\beta u   \|_{L^q(\R^d)}^{(1-a)}. 
\end{equation*}

\end{enumerate}
Assume in addition that $\alpha  - \sigma \le 1$ and $\rr > 1$. We have
\begin{enumerate}
\item[A3)] if $1/ \rr + \gamma/ d = 0$ and $\supp u \subset B_R$ for some $R>0$,   then 
\begin{equation*}
\left(\int_{\R^d}  \frac{ |x|^{\gamma \rr}}{\ln^\rr(2 R/ |x|)} |u|^{\rr} \,dx \right)^{1/\rr} \le C \| |x|^\alpha \nabla u \|_{L^p(\R^d)}^{a}\| |x|^\beta u   \|_{L^q(\R^d)}^{(1-a)}, 
\end{equation*}

\item[A4)] if $1/ \rr + \gamma/ d = 0$ and  $\supp u \subset \R^d \setminus B_r$ for some $r>0$,  then 
\begin{equation*}
\left(\int_{\R^d}  \frac{ |x|^{\gamma \rr}}{\ln^\rr(2 |x| / r )} |u|^{\rr} \, dx \right)^{1/\rr} \le C \| |x|^\alpha \nabla u \|_{L^p(\R^d)}^{a}\| |x|^\beta u   \|_{L^q(\R^d)}^{(1-a)}. 
\end{equation*}

\end{enumerate}
Here  $C$  denotes a positive constant independent of $u$, $r$, and $R$. 
\end{proposition}

%\begin{remark}\rm In comparison with \eqref{ineq-CKN} with Proposition~\ref{pro1}, beside covering all the possibilities of $1/ \rr + \gamma/ d$, Proposition~\ref{pro1} does not impose the conditions $1/ p + \alpha/ d > 0$ and $1/ q + \beta / d > 0$ as required  in \eqref{ineq-CKN}. Note that the RHS in the assertions in $A1), \ A2, \ A3), $ and $A4)$ is finite for $u \in C^1_{c}(\R^d \setminus \{0 \})$. 
%\end{remark}

Assertion $A1)$ is a slight improvement of 
the classical Caffarelli-Kohn-Nirenberg. Indeed, in the classical setting, Assertion $A1)$ is established under the {\it additional} assumptions 
\begin{equation*}
%\label{CKN-assumption}
1/ p + \alpha/ d > 0 \quad \mbox{ and } \quad 1/ q + \beta/ d > 0,
\end{equation*}
as mentioned  in \eqref{totototo} in the introduction.  Assertion $A2)$ with $a = 1$ and $\rr = p$ was  known (see, e.g., \cite{FS}).  Concerning Assertion $A3)$ with $a = 1$, this was obtained for $d=2$ in  \cite{CM} and \cite{adi} and, for $d \ge 3$, this was established in \cite{adi}. Assertion $A4)$ with $a=1$ might be known; however, we cannot find any references for it. To our knowledge, the remaining cases seem to be new. 

\medskip 
Analogous versions in a bounded domain will be given in Section~\ref{boundedcase}. 

\medskip 

The ideas used in the proof of Theorems~\ref{thm-Hardy} and \ref{thm-CKN}, and their general version (Theorem~\ref{CKN-g}) are as follows. On one hand, this is based on Poincare's and Sobolev inequalities related to $I_\delta(u, \Omega)$ (see Lemma~\ref{techlemma}  and Lemma~\ref{lem-S}). These inequalities have their roots in \cite{Ng11}. Using these inequalities, we derive 
the key estimate (see Lemma~\ref{lem-S-P-3} and also Lemma~\ref{techlemma}), for an annulus $D$ centered at the origin and for $\lambda > 0$,  
\begin{equation}\label{Poincare-1}
\left(\dashint_{\lambda D} \left| u -  \dashint_{\lambda D} u \right|^{\rr} \, dx  \right)^{1/\rr} \le C  \Big(\lambda^{p-d} I_\delta (u, \lambda D) + \delta^p \Big)^{a/p} 
\left(\dashint_{\lambda D} \left| u -  \dashint_{\lambda D} u \right|^{q} \, dx  \right)^{(1-a)/q}, 
\end{equation}
for some positive constant  $C$ independent of $u$ and $\lambda$. 
On the other hand, decomposing $\R^d$ into annuli $\C_k$ which are defined by  
\begin{equation*}
\C_k : = \big\{x \in \R^d: 2^ k \le |x| < 2^{k+1} \big\},
% \quad \mbox{ and } \quad \C_k : = \big\{x \in \R^3; 2^{k} \le |x| < 2^{k+2} \big\}.
\end{equation*}
and applying \eqref{Poincare-1}  to each $\C_k$, we obtain 
$$
\left(\dashint_{\C_k} \left| u - \dashint_{\C_k} u  \right|^\rr \, dx \right)^{1/\rr} \le C \Big(2^{-(d-p) k} I_\delta(u, \C_k) + \delta^p \Big)^{a/p}  \Big( \dashint_{\C_k}|u|^q\Big)^{(1-a)/q}, 
$$
Similar idea was used in \cite{CKN}. Using \eqref{Poincare-1} again in the cases $i)$ and $ii)$, we can derive an appropriate estimate for 
\begin{equation*}
2^{(\gamma \rr + d)k}\left| \dashint_{\C_{k}} u \right|^\rr. 
\end{equation*}
This is the novelty in comparison with the approach in \cite{CKN}. Combining these two facts, one obtains the desired inequalities. The other cases follow similarly.  Similar approach is used to establish Caffarelli-Kohn-Nirenberg's inequalities for fractional Sobolev spaces in \cite{NgS3}. 

We now make some comments on the magnetic Sobolev setting. 
If $A:\R^d\to\R^d$ is locally bounded and $u:\R^d\to{\mathbb C}$, we set
$$
\Psi_u(x,y):=e^{\i (x-y)\cdot A\left(\frac{x+y}{2}\right)}u(y),\quad\,\, x,y\in\R^d.
$$
The following {\em diamagnetic inequality} holds 
$$
||u(x)|-|u(y)||\leq
\big|\Psi_u(x,x)-\Psi_u(x,y)\big|,\quad\text{for a.e.\ $x,y\in\R^d$.}
$$	
In turn, by defining
\begin{align*}
I_{\delta}^A(u, \alpha) &= \mathop{\int_{\R^d} \int_{\R^d}}_{\{|\Psi_u(x,y)-\Psi_u(x,x)|>\delta \}} \frac{\delta^p |x|^{p \alpha}}{|x - y|^{d + p}} \, dx \, dy, 
%J_{\delta}^A(u, \alpha) &= \mathop{\int_{\R^d} \int_{\R^d}}_{\{|\Psi_u(x,y)-\Psi_u(x,x)|>\delta \}} \frac{\delta^p|x|^{p\alpha}}{|x - y|^{d + p}} \, dx \, dy, 
\end{align*}
we have, for $\alpha \in \R$, 
\begin{equation*}
\label{dia}
I_{\delta}(|u|, \alpha) \leq 	I_\delta^A(u, \alpha ) \quad \text{for all $\delta>0$.}
\end{equation*}
Then, the assertions of Theorem~\ref{thm-Hardy} and \ref{thm-CKN} 
keep holding with $ I_{\delta}^A(u, 0)$ (resp.\ $I_{\delta}^A(u, \alpha)$) on the right-hand side in place of $ I_{\delta}(u)$ (resp.\ $I_{\delta}(u, \alpha)$). 
For the sake of completeness, we refer the reader to \cite{magn-case} 
for some recent results about new characterizations of classical magnetic Sobolev spaces in the terms of $I_\delta^A(u, 0)$ (see \cite{magn-case, squ-volz, acv-p}  for the ones related to $J_\delta$).
\medskip

The paper is organized as follows. \newline
In Section~\ref{hardy} we prove Theorem~\ref{thm-Hardy}.
In Section~\ref{ckn} we prove Theorem~\ref{CKN-g} and Proposition~\ref{CKN-g-1} which imply Theorem~\ref{thm-CKN} and Proposition~\ref{pro1}. In Section~\ref{boundedcase} we present  
versions of Theorems~\ref{thm-Hardy} and \ref{CKN-g} in a bounded domain $\Omega$. 

\section{Improved Hardy's inequality} 
\label{hardy}
We first recall that a straightforward variant of \cite[Theorem 1]{Ng11} yields the following
%\footnote{\cite[Theorem 1]{Ng11} states for a unit ball; 
%	however, similar proof holds for the result stated here.}, 
\begin{lemma}
	\label{techlemma}
Let $d\geq 1$, $p\geq 1$  and set
$$
D: =\big\{ x \in \R^d: r < |x| < R \big\}.
$$
Then 
\begin{equation*}
%\label{inequality-0}
 \dashint_{D} \left|u(x) - \dashint_{D} u\right|^p \, dx \le C_{r, R} \big( I_\delta(u, D) + \delta^p \big), \quad\mbox{for all $u \in L^p(D)$}. 
\end{equation*}
 As a consequence, we have, for  $\lambda > 0$, 
\begin{equation}\label{inequality-1}
 \dashint_{\lambda D} \left|u(x) - \dashint_{\lambda D} u\right|^p \, dx \le C_{r, R} \big( \lambda^{p-d} I_\delta(u, \lambda D) + \delta^p \big), \quad\mbox{for all   $u \in L^p(\lambda D)$}, 
\end{equation}
where  $\lambda D: = \{\lambda x: x \in D \}$. Here  $C_{r, R}$ denotes a positive constant independent of $u$, $\delta$, and $\lambda$. 
\end{lemma}

The following elementary inequality will be used several times in this paper. 

\begin{lemma}\label{lem-Holder} Let  $\Lambda> 1$ and $\rr > 1$.  There exists $C = C(\Lambda, \rr) > 0$, depending only on $\Lambda$ and $\rr$ such that, for all $1 < c < \Lambda$, 
\begin{equation}\label{thm1-observation}
(|a| + |b|)^\rr \le c |a|^\rr + \frac{C}{(c - 1)^{\rr -1}} |b|^\rr,\quad \mbox{ for all } a, b \in \R.  
\end{equation}
\end{lemma}

\begin{proof} Since \eqref{thm1-observation} is clear in the case $|b| \ge |a|$ and in the case $b=0$, by rescaling and considering $x = |a|/ |b|$, it suffices to prove, for $C = C (\Lambda, \rr)$  large enough, that
\begin{equation}\label{lem-p0}
(x + 1)^\rr \le c x^\rr + \frac{C}{(c - 1)^{\rr -1}},\quad \mbox{ for all } x  \ge 1.   
\end{equation}
Set 
$$
f(x) = (x + 1)^\rr - c x^\rr - \frac{C}{(c - 1)^{\rr -1}} \mbox{ for } x > 0. 
$$
We have 
$$
f'(x) = \rr (x+1)^{\rr -1} - c \rr x^{\rr -1} \quad \mbox{ and } \quad f'(x) = 0 \mbox{ if and only if } x = x_0 :=  \big(c^{\frac{1}{\rr -1} } - 1\big)^{-1}.
$$ 
One can check that
\begin{equation}\label{lem-p1}
\lim_{x \to + \infty} f(x) = - \infty,  \quad  \lim_{x \to 1} f(x) < 0 \mbox{ if $C = C(\Lambda, \rr)$ is large enough}.   
\end{equation}
and 
\begin{equation}\label{lem-p2}
f(x_0) = c x_0^{\rr -1}  - \frac{C}{(c - 1)^{\rr -1}}.
\end{equation}
If $c^{\frac{1}{\rr -1}} > 2$ then $x_0 < 1$ and \eqref{lem-p0} follows from \eqref{lem-p1}. Otherwise $1 \le s : = c^{\frac{1}{\rr -1}} \le 2 $. By the mean value theorem, we have
$$
s^{\rr -1} - 1 \le  (s - 1) \max_{1 \le t \le 2} (\rr -1) t^{\rr -2} \mbox{ for } 1 \le s \le 2. 
$$
We derive from \eqref{lem-p2} that, with $C = \Lambda \big[\max_{1 \le t \le 2} (\rr -1) t^{\rr -2} \big]^{\rr-1}$,  
$$
f(x_0) < 0. 
$$
The conclusion now follows from \eqref{lem-p1}. 
\end{proof}

%\textcolor{red}{Some comments on the proof!}

%,  and 
%\begin{equation}\label{inequality-2}
% \dashint_{B_1} |u(x) - \dashint_{B_1} u|^2 \, dx \le C \Big( I_\delta(u, B_1) + \delta^2 \Big) \mbox{ for all } u \in L^2(B_1). 
%\end{equation}
%\noindent
We are now ready to give 
\medskip 

\noindent{\bf Proof of  Theorem~\ref{thm-Hardy}.}   Let $m, n \in \Z$ be such that 
$$
2^{n-1} \le R < 2^n \quad \mbox{ and } \quad 2^{m} \le  r < 2^{m + 1}.  
$$
It is clear that $n - m \ge 1$.
By \eqref{inequality-1} of Lemma~\ref{techlemma}, we have, for all $k \in \Z$,  
\begin{equation*}
\dashint_{\C_k} \left|u(x) - \dashint_{\C_k} u\right|^p \, dx \le C \Big( 2^{-(d-p)k} I_\delta(u, \C_k) + \delta^p \Big).   
\end{equation*}
Here and in what follows in this proof, $C$ denotes a positive constant independent of $k$, $u$, and $\delta$. This implies 
\begin{equation*}
2^{-pk} \int_{\C_k} \left|u(x) - \dashint_{\C_k} u\right|^p \, dx \le C \Big(  I_\delta(u, \C_k) + 2^{(d-p)k} \delta^p \Big). 
\end{equation*}
It follows that 
\begin{equation}\label{thm1-part1}
2^{-pk} \int_{\C_k} |u(x)|^p \, dx \le  C 2^{(d-p)k}  \left|\dashint_{\C_k} u\right|^p  +  C \Big(  I_\delta(u, \C_k) + 2^{(d-p)k} \delta^p \Big).
\end{equation}

\noindent
$\bullet$ {\bf Step 1:} Proof of $i)$.  Summing \eqref{thm1-part1} with respect to $k$ from $-\infty$ to $n$, we obtain 
\begin{equation}\label{thm1-part2}
\int_{\R^d} \frac{|u(x)|^p}{|x|^p} \, dx \le C \sum_{k = -\infty}^n 2^{(d-p)k} \left|\dashint_{\C_k} u\right|^p   + C I_{\delta} (u) + C  2^{(d-p)n}  \delta^p  ,  
\end{equation}
since $d>p$. We also have, by \eqref{inequality-1}, for $k \in \Z$, 
\begin{equation*}
\left| \dashint_{\C_{k}} u - \dashint_{\C_{k+1}} u \right| \le  C \Big(2^{-(d-p)k}  I_\delta(u, \C_k \cup \C_{k+1}) + \delta^p \Big)^{1/p}.
\end{equation*}
This implies 
\begin{equation*}
\left| \dashint_{\C_{k}} u \right| \le \left| \dashint_{\C_{k+1}} u \right| +   C \Big(2^{-(d-p)k}  I_\delta(u, \C_k \cup \C_{k+1}) + \delta^p \Big)^{1/p}.
\end{equation*}
Applying Lemma~\ref{lem-Holder}, we have 
\begin{equation*}
\left| \dashint_{\C_{k}} u \right|^p \le \frac{ 2^{d-p+1}}{1 + 2^{d-p}} \left|\dashint_{\C_{k+1}} u \right|^p +   C \Big(2^{-(d-p)k}  I_\delta(u, \C_k \cup \C_{k+1}) + \delta^p \Big).
\end{equation*}
It follows that, with  $c = 2/ (1 + 2^{d-p}) < 1$, 
\begin{equation*}
2^{(d-p)k}\left| \dashint_{\C_{k}} u \right|^p \le  c 2^{(d-p)(k+1)}\left| \dashint_{\C_{k+1}} u \right|^p + C \Big(  I_\delta(u, \C_k \cup \C_{k+1}) + 2^{(d-p)k} \delta^p \Big). 
\end{equation*}
We derive that 
\begin{equation}\label{thm1-part3}
\sum_{k = -\infty}^n 2^{(d-p)k} \left|\dashint_{\C_k} u\right|^p \le C \sum_{k=-\infty}^n I_{\delta} (u, \C_k \cup \C_{k+1}) + C 2^{(d-p)n} \delta^p. 
\end{equation}
A combination of \eqref{thm1-part2} and \eqref{thm1-part3} yields 
\begin{equation*}
\int_{\R^d} \frac{|u(x)|^d}{|x|^d} \, dx \le  C I_{\delta} (u) + C  2^{(d-p)n} \delta^p. 
\end{equation*}
The conclusion  of $i)$ follows.

\medskip 
\noindent
$\bullet$ {\bf Step 2:}  Proof of $ii)$.  Summing \eqref{thm1-part1} with respect to $k$ from $m$ to $+ \infty$, we obtain 
\begin{equation}\label{thm1-part2-4}
\int_{\R^d} \frac{|u(x)|^p}{|x|^p} \, dx \le C \sum_{k = m}^{+\infty} 2^{(d-p)k} \left|\dashint_{\C_k} u\right|^p   + C I_{\delta} (u) + C  2^{(d-p)m}  \delta^p  ,  
\end{equation}
since $p>d$. We also have, by \eqref{inequality-1}, for $k \in \Z$, 
\begin{equation*}
\left| \dashint_{\C_{k}} u - \dashint_{\C_{k+1}} u \right| \le  C \Big(2^{-(d-p)k}  I_\delta(u, \C_k \cup \C_{k+1}) + \delta^p \Big)^{1/p}.
\end{equation*}
This implies that 
\begin{equation*}
\left| \dashint_{\C_{k+1}} u \right| \le \left| \dashint_{\C_{k}} u \right| +   C \Big(2^{-(d-p)k}  I_\delta(u, \C_k \cup \C_{k+1}) + \delta^p \Big)^{1/p}.
\end{equation*}
Applying Lemma~\ref{lem-Holder}, we have 
\begin{equation*}
\left| \dashint_{\C_{k+1}} u \right|^p \le \frac{1 + 2^{d-p}}{2^{d-p+1}}  \left|\dashint_{\C_{k}} u \right|^p +   C \Big(2^{-(d-p)k}  I_\delta(u, \C_k \cup \C_{k+1}) + \delta^p \Big).
\end{equation*}
It follows that, with  $c = (1 + 2^{d-p})/2 < 1$, 
\begin{equation*}
2^{(d-p)(k+1)}\left| \dashint_{\C_{k+1}} u \right|^p \le  c 2^{(d-p)k}\left| \dashint_{\C_{k}} u \right|^p + C \Big(  I_\delta(u, \C_k \cup \C_{k+1}) + 2^{(d-p)k} \delta^p \Big).
\end{equation*}
We derive that 
\begin{equation}\label{thm1-part3-4}
\sum_{k = m}^{+\infty} 2^{(d-p)k} \left|\dashint_{\C_k} u\right|^p \le C  I_{\delta} (u) + C 2^{(d-p)m} \delta^p. 
\end{equation}
A combination of \eqref{thm1-part2-4} and \eqref{thm1-part3-4} yields 
\begin{equation*}
\int_{\R^d} \frac{|u(x)|^p}{|x|^p} \, dx \le  C I_{\delta} (u) + C  2^{(d-p)m} \delta^p. 
\end{equation*}
The conclusion  of $ii)$ follows.

\medskip 
\noindent
$\bullet$ {\bf Step 3:} Proof of $iii)$.  Let $\alpha > 0$.  Summing \eqref{thm1-part1} with respect to $k$ from $m$ to $n$, we obtain 
\begin{equation}\label{thm1-part2-2}
\int_{\{ 2^m < |x| <  2^{n} \}} \frac{|u(x)|^d}{|x|^d \ln^{\alpha + 1} (2 R / |x|)} \, dx \le C  \sum_{k = m}^n \frac{1}{(n - k + 1)^{\alpha + 1}} \left|\dashint_{\C_k} u\right|^d   + C I_{\delta} (u) + C (n-m)  \delta^d  . 
\end{equation}
 We also have, by \eqref{inequality-1}, for $k \in \Z$, 
\begin{equation}\label{thm1-part3-2}
\left| \dashint_{\C_{k}} u \right|  \le \left|  \dashint_{\C_{k+1}} u \right| + C \Big( I_\delta(u, \C_k \cup \C_{k+1})^{1/d} + \delta \Big).
\end{equation}
By applying  Lemma~\ref{lem-Holder} with
$$
c = \frac{(n-k+1)^\alpha}{(n-k+1/2)^\alpha}, 
$$
it follows from \eqref{thm1-part3-2}  that, for $m \le k \le n$,  
\begin{align}\label{thm1-part4-2}
 \frac{1}{(n - k + 1)^{\alpha}}  \left| \dashint_{\C_{k}} u \right|^d \le &   \frac{1}{(n - k + 1/2)^{\alpha}} \left| \dashint_{\C_{k+1}} u \right|^d \\[6pt] 
 &+ C (n - k + 1)^{d-1 - \alpha} \Big(  I_\delta(u, \C_k \cup \C_{k+1}) +  \delta^d \Big) \nonumber. 
\end{align}
We have, $m \le k \le n$, 
\begin{equation}\label{thm1-part5-2}
\frac{1}{(n - k + 1)^{\alpha}} - \frac{1}{(n - k + 3/2)^{\alpha}} \sim \frac{1}{(n - k + 1)^{\alpha + 1}}. 
\end{equation}
Taking  $\alpha = d-1$ and   combining  \eqref{thm1-part4-2} and \eqref{thm1-part5-2} yield
\begin{equation}\label{thm1-part6-2}
\sum_{k = m}^n \frac{1}{(n - k + 1)^{d}}  \left| \dashint_{\C_{k}} u \right|^d \le C I_{\delta} (u) + C  (n - m) \delta^d. 
\end{equation}
From \eqref{thm1-part2-2} and \eqref{thm1-part6-2}, we obtain 
\begin{equation*}
\int_{\{ |x| > 2^{m} \}} \frac{|u(x)|^d}{|x|^d \ln^{d} (2R/|x|)} \, dx \le  C I_{\delta} (u) + C (n - m) \delta^d. 
\end{equation*}
This implies the conclusion of $iii)$.

% for $R = 2^{k_0}$. The conclusion  in the general case follows by applying the previous case to the function $v(\cdot ) := u( 2^{k_0} \cdot / R)$. 

%By taking $\alpha = \beta -1 > 2$,  we derive from  \eqref{thm1-part4-2} and \eqref{thm1-part5-2} that
%\begin{equation*}
%\sum_{k = k_0}^m \frac{1}{(m - k + 1)^{\alpha}}  \Big| \dashint_{\C_{k}} u \Big|^d \le C I_{\delta} (u) + C \delta^d, 
%\end{equation*}
%which yields, by  \eqref{thm1-part2-2}, 
%\begin{equation*}
%\int_{|x| > 2^{k_0}} \frac{|u(x)|^d}{|x|^d \ln^{\beta} (2R/|x|)} \, dx \le  C I_{\delta} (u) + C \delta^d. 
%\end{equation*}
%Letting $R \to 0_+$, i.e., $k_0 \to - \infty$, one obtains the second part of $ii)$.  

\medskip
\noindent
$\bullet$ {\bf Step 4} Proof of $iv)$. Let $\alpha > 0$.  Summing \eqref{thm1-part1} with respect to $k$ from $m$ to $n$, we obtain 
\begin{equation}\label{thm1-part2-3}
\int_{\{2^m < |x| <  2^{n}\}} \frac{|u(x)|^d}{|x|^d \ln^{\alpha + 1} (2 |x| / R)} \, dx \le C  \sum_{k = m}^n \frac{1}{(k-m + 1)^{\alpha + 1}} \left|\dashint_{\C_k} u\right|^d   + C I_{\delta} (u) + C   \delta^d  . 
\end{equation}
We have, by \eqref{inequality-1}, for $k \in \Z$, 
\begin{equation}\label{thm1-part3-3}
\left| \dashint_{\C_{k+1}} u \right|  \le \left|  \dashint_{\C_{k}} u \right| + C \Big( I_\delta(u, \C_k \cup \C_{k+1})^{1/d} + \delta \Big).
\end{equation}
By applying  Lemma~\ref{lem-Holder} with
$$
c = \frac{(n-k+1)^\alpha}{(n-k+1/2)^\alpha}, 
$$
it follows from   \eqref{thm1-part3-3} that, for $m \le k + 1 \le n$,  
\begin{align}\label{thm1-part4-3}
 \frac{1}{(k - m + 1)^{\alpha}}  \Big| \dashint_{\C_{k+1}} u \Big|^d & \le  \frac{1}{(k - m + 1/2)^{\alpha}} \Big| \dashint_{\C_{k}} u \Big|^d  \\
 &+ C (k- m + 1)^{d-1 - \alpha} \Big(  I_\delta(u, \C_k \cup \C_{k+1}) +  \delta^d \Big). \notag
\end{align}
We have, $m \le k + 1 \le n$, 
\begin{equation}\label{thm1-part5-3}
\frac{1}{(k - m + 1)^{\alpha}} - \frac{1}{(k - m + 3/2)^{\alpha}} \sim \frac{1}{(k - m + 1)^{\alpha + 1}}. 
\end{equation}
Taking  $\alpha = d-1$ and   combining  \eqref{thm1-part4-3} and \eqref{thm1-part5-3} yield
\begin{equation}\label{thm1-part6-3}
\sum_{k = m}^n \frac{1}{(k - m + 1)^{d}}  \Big| \dashint_{\C_{k}} u \Big|^d \le C I_{\delta} (u) + C  (n - m) \delta^d. 
\end{equation}
From \eqref{thm1-part2-3} and \eqref{thm1-part6-3}, we obtain 
\begin{equation*}
\int_{\{ 2^m < |x| <  2^{n} \}} \frac{|u(x)|^d}{|x|^d \ln^{d} (2|x|/ R)} \, dx \le  C I_{\delta} (u) + C (n - m) \delta^d. 
\end{equation*}
This implies the conclusion of $iv)$.
 
 \medskip 
The proof is complete. \qed

\section{Improved Caffarelli-Kohn-Nirenberg's inequality} 
\label{ckn}

In the proof of Theorem~\ref{thm-CKN}, we use the following result

\begin{lemma}\label{lem-S} Let $1 < p < d$, $\Omega$ be a smooth bounded open subset of $\R^d$,  and $v \in L^p(\Omega)$. We have 
\begin{equation*}
\| u \|_{L^{p^*}(\Omega)} \le C_\Omega \Big( I_\delta (u)^{1/p} + \|  u \|_{L^p} + \delta \Big),  
\end{equation*}
where $p^*: = d p / (d-p)$ denotes the Sobolev exponent of $p$. 
\end{lemma}

\begin{proof} For $\tau > 0$, let us set 
$$
\Omega_{\tau}: = \big\{x \in \R^d: \, \mbox{dist} (x, \Omega) < \tau \big\}. 
$$
Since $\Omega$ is smooth, by \cite[Lemma 17]{BHN}, there exists $\tau>0$ small enough and an extension $U$ of $u$ in $\Omega_\tau$ such that 
\begin{equation}\label{lem-S-extension}
I_\delta(U, \Omega_\tau) \le C I_{\delta} (u, \Omega) \quad \mbox{ and } \quad \| U \|_{L^p(\Omega_\tau)} \le C \| u \|_{L^p(\Omega)}, 
\end{equation}
for $0 < \delta < 1$. 
Fix such a $\tau$. Let $\varphi \in C^1(\R^d)$ such that 
$$
{\rm supp} \,\varphi \subset \Omega_{2\tau/3},\qquad \text{$\varphi = 1$\, in $\Omega_{\tau/3}$},  \qquad \text{$0 \le \varphi \le 1$\, in $\R^d$}. 
$$
Define $v = \varphi  U \mbox{ in } \R^d.$ We claim that 
\begin{equation}\label{lem-S-claim}
I_{2 \delta} (v) \le C \Big( I_\delta (u, \Omega) + \| u\|_{L^p(\Omega)}^p \Big). 
\end{equation}
Indeed, set
$$
f(x, y) = \frac{\delta^p}{|x-y|^{d + p}} \mathds{1}_{\{|v(x) - v(y)| > 2 \delta \}}. 
$$
We  estimate $I_{2\delta}(v)$. 
We have 
\begin{equation*}
\iint_{\Omega \times \R^d} f(x, y) \, dx \, dy   \le \iint_{\Omega_{\tau/3}\times \Omega_{\tau/3}} f(x, y) \, dx \, dy  +  \mathop{\iint_{\Omega_{\tau}  \times \R^d}}_{\{|x - y| > \tau/4\}} f(x, y) \, dx \, dy,
\end{equation*}
and, since $v = 0$ in $\Omega_\tau \setminus \Omega_{2 \tau /3}$, 
\begin{equation*}
\iint_{(\R^d \setminus \Omega_\tau) \times \R^d} f(x, y) \, dx \, dy  \leq
 \iint_{(\R^d \setminus \Omega_{\tau}) \times (\R^d \setminus \Omega_{\tau})}  f(x, y) \, dx \, dy  +   \mathop{\iint_{\Omega_{\tau}  \times \R^d}}_{\{|x - y| > \tau/4\}} f(x, y) \, dx \, dy, 
\end{equation*}
\begin{multline*}
\iint_{(\Omega_\tau \setminus \Omega) \times \R^d} f(x, y) \, dx \, dy \le    \iint_{(\Omega_{\tau} \setminus \Omega) \times (\Omega_{\tau} \setminus \Omega)} f(x, y) \, dx \, dy  \\[6pt]   + \iint_{\Omega_{\tau/3}\times \Omega_{\tau/3}} f(x, y) \, dx \, dy + \mathop{\iint_{\Omega_{\tau}  \times \R^d}}_{\{|x - y| > \tau/4\}} f(x, y) \, dx \, dy.
\end{multline*}
%
%We have
%\begin{align*}
%\iint_{\R^d \times \R^d} f(x, y) \, dx \, dy  &\le \iint_{\Omega_{\tau/3}\times \Omega_{\tau/3}} f(x, y) \, dx \, dy  + \iint_{(\Omega_{\tau} \setminus \Omega) \times (\Omega_{\tau} \setminus \Omega)} f(x, y) \, dx \, dy \\ 
%&+ \iint_{(\R^d \setminus \Omega_{\tau}) \times (\R^d \setminus \Omega_{\tau})}  f(x, y) \, dx \, dy  + 3  \mathop{\iint_{\Omega_{\tau}  \times \R^d}}_{\{|x - y| > \tau/4\}} f(x, y) \, dx \, dy. 
%\end{align*}
It is clear that, by \eqref{lem-S-extension},  
\begin{equation}\label{lem-S-part1}
\iint_{\Omega_{\tau/3}\times \Omega_{\tau/3}}  f(x, y) \, dx \, dy \le C I_\delta(u, \Omega), 
\end{equation}
by the fact that $\varphi = 0 $ in $\R^d \setminus \Omega_\tau$, 
\begin{equation}\label{lem-S-part2}
\iint_{(\R^d \setminus \Omega_{\tau}) \times (\R^d \setminus \Omega_{\tau})}  f(x, y) \, dx \, dy  = 0, 
\end{equation}
and,  by a straightforward computation, 
\begin{equation}\label{lem-S-part3}
\mathop{\iint_{\Omega_{\tau}  \times \R^d}}_{\{|x - y| > \tau/4\}} f(x, y) \, dx \, dy \le  C \delta^p. 
\end{equation}
%
%\begin{equation}\label{lem-S-part3}
%\mathop{\iint_{\Omega_{\tau}  \times \R^d}}_{\{|x - y| > \tau/4\}} f(x, y) \, dx \, dy \le \int_{\Omega_\tau} \, dy \int_{ \{ |x-y| > C \delta / |U(y)|\}} \frac{\delta^p}{|x-y|^{d + p}} \, dx  =   C \int_{\Omega_\tau} |U(y)|^p \, dy. 
%\end{equation}
We have, for $x, y \in \R^d$,  
$$
v(x) - v(y) = \varphi(x) \big(U(x) - U(y) \big) + U(y) \big(\varphi(x) - \varphi (y) \big). 
$$
It follows that if $|v(x) - v(y)| > 2 \delta$ then either 
$$ 
|U(x) - U(y)|  \ge  |\varphi(x) \big(U(x) - U(y) \big)| > \delta 
$$ 
or 
$$
C |U(y)| |x-y| \ge 
|U(y) \big(\varphi(x) - \varphi (y) \big)| > \delta. 
$$
We thus derive that 
\begin{align}\label{lem-S-part4}
 \iint_{(\Omega_{\tau} \setminus \Omega) \times (\Omega_{\tau} \setminus \Omega)} f(x, y) \, dx \, dy &\le \mathop{\int_{(\Omega_{\tau} \setminus \Omega) } \int_{(\Omega_{\tau} \setminus \Omega) }}_{\{|U(x) - U(y)| > \delta \}} \frac{\delta^p}{|x - y|^{d + p}} \, dx \, dy \\
& + 
 \mathop{\int_{(\Omega_{\tau} \setminus \Omega) } \int_{(\Omega_{\tau} \setminus \Omega) }}_{\{|x -y| > C \delta/ |U(y)| \}} \frac{\delta^p}{|x - y|^{d + p}} \, dx \, dy.  \notag
 \end{align}
A straightforward computation yields 
\begin{equation*}
\mathop{\int_{(\Omega_{\tau} \setminus \Omega) } \int_{(\Omega_{\tau} \setminus \Omega) }}_{\{|x -y| > C \delta/ |U(y)| \}} \frac{\delta^p}{|x - y|^{d + p}} \, dx \, dy  \le \int_{\Omega_\tau} \, dy \int_{ \{ |x-y| > C \delta / |U(y)|\}} \frac{\delta^p}{|x-y|^{d + p}} \, dx  =  C \int_{\Omega_\tau} |U(y)|^p \, dy. 
\end{equation*}
Using \eqref{lem-S-extension}, we deduce  from \eqref{lem-S-part4} that 
\begin{equation}\label{lem-S-part6}
 \iint_{(\Omega_{\tau} \setminus \Omega) \times (\Omega_{\tau} \setminus \Omega)} f(x, y) \, dx \, dy \le  C I_\delta(u, \Omega) + C \| u \|_{L^p(\Omega)}^p 
 \end{equation}
A combination of \eqref{lem-S-part1}, \eqref{lem-S-part2}, \eqref{lem-S-part3}, and \eqref{lem-S-part6} yields Claim~\eqref{lem-S-claim}. 
By applying \cite[Theorem 3]{Ng11} and using the fact $\supp v \subset \Omega_\tau$, we have 
\begin{equation}\label{p>1}
\| v\|_{L^{p^*}(\R^d)} \le C I_{2 \delta} (v)^{1/p} + C \delta.  
\end{equation}
The conclusion now follows from Claim~\eqref{lem-S-claim}. 
\end{proof}

\begin{remark} \rm The assumption $p>1$ is required in \eqref{p>1}.
\end{remark}

As a consequence of Lemmas~\ref{techlemma} and \ref{lem-S}, we obtain 
\begin{corollary}\label{cor-S-P-1}
Let $d\geq 2$, $1 <   p <  d$, $0 < r < R$, and $\lambda > 0$,  and set 
$$
\lambda D: =\big\{\lambda x \in \R^d: r  < |x| < R \big\}.
$$
We have, for $1 \le q \le p^*$,  
\begin{equation*}
\left( \dashint_{\lambda D} \left|u(x) - \dashint_{\lambda D} u\right|^q \, dx \right)^{1/q} \le C_{r, R} \Big( \lambda^{p-d} I_\delta(u, \lambda D) + \delta^p \Big)^{1/p}, \quad\mbox{for $u \in L^p(\lambda D)$}, 
\end{equation*}
where   $C_{r, R}$ denotes a positive constant independent of $u$, $\delta$, and $\lambda$. 
\end{corollary}

%As a consequence of John-Nirenberg's inequality \cite{JN} and Lemma~\ref{techlemma}, we also have
%\begin{corollary}\label{cor-S-P-2}
%Let $d\geq 1$, $p \ge d$, $0 < r < R$, and $\lambda > 0$,  and set 
%$$
%\lambda D: =\big\{\lambda x \in \R^d: r  < |x| < R \big\}.
%$$
%We have, for $s > 0$,  
%\begin{equation*}
%\Big( \dashint_{\lambda D} \left|u(x) - \dashint_{\lambda D} u\right|^s \, dx \Big)^{1/s} \le C_{r, R} \big( \lambda^{p-d} I_\delta(u, \lambda D) + \delta^p \big)^{1/p} \quad\mbox{for $u \in L^1(\lambda D)$}, 
%\end{equation*}
%where  $C_{r, R}$ is a positive constant independent of $u$, $\delta$, and $\lambda$. 
%\end{corollary}

%As a consequence of Corollaries~\ref{cor-S-P-1} and \ref{cor-S-P-2}, we get
%\begin{corollary}\label{cor-S-P-3} Let $d \ge 1$,  $p>1$ and $q \ge 1$ be such that 
%\begin{equation*}
%1/q \ge 1/p - 1/ d. 
%\end{equation*}
%Set 
%$$
%\lambda D: =\big\{\lambda x \in \R^d: R  < |x| < R \big\}.
%$$
%We have, for $1 \le q \le p^*$ \footnote{As usual, as $p \ge d$, $p^*: = + \infty$ and this inequality reads $1 \le q <  + \infty$.},  
%\begin{equation*}
%\Big( \dashint_{\lambda D} \left|u(x) - \dashint_{\lambda D} u\right|^q \, dx \Big)^{1/q} \le C_{R, R} \big( \lambda^{p-d} I_\delta(u, \lambda D) + \delta^p \big)^{1/p}, \quad\mbox{for all   $u \in L^1(\lambda D)$}, 
%\end{equation*}
%\end{corollary}

Here is an application of Corollaries~\ref{cor-S-P-1}  which plays a crucial role in the proof of Theorem~\ref{CKN-g} below.

\begin{lemma}\label{lem-S-P-3} Let $d \ge 1$, $1 < p <  d$, $q \ge 1$, $\rr > 0$, and $0 \le a \le 1$ be such that
$$
\frac{1}{\rr} \ge a \left(\frac{1}{p} - \frac{1}{d} \right) + \frac{1-a}{q}. 
$$
Let $0 < r < R$, and $\lambda > 0$  and set 
$$
\lambda D: =\big\{\lambda x \in \R^d: r  < |x| < R \big\}.
$$
Then, for $u \in L^1(\lambda D)$, 
\begin{equation*}
%\label{Poincare}
\left(\dashint_{\lambda D} \Big| u -  \dashint_{\lambda D} u \Big|^{\rr} \, dx  \right)^{1/\rr} \le C  \Big(\lambda^{p-d} I_\delta (u, \lambda D) + \delta^p \Big)^{a/p} 
\left(\dashint_{\lambda D} \left| u -  \dashint_{\lambda D} u \right|^{q} \, dx  \right)^{(1-a)/q} , 
\end{equation*}
for some positive constant $C$ independent of $u$, $\lambda$, and $\delta$. 
\end{lemma}

\begin{proof}
Let $\rr, \sigma, \, t  > 0$, be such that 
$$
\frac{1}{\rr} \ge \frac{a}{\sigma} + \frac{1-a}{t}. 
$$
We have, by  a standard interpolation inequality, that
\begin{equation*}
\left(\dashint_{\lambda D} \Big| u -  \dashint_{\lambda D} u \Big|^{\rr} \, dx  \right)^{1/\rr} \le \left(\dashint_{\lambda D} \Big| u -  \dashint_{\lambda D} u \Big|^{\sigma} \, dx  \right)^{a/\sigma} 
\left(\dashint_{\lambda D} \Big| u -  \dashint_{\lambda D} u \Big|^{t} \, dx  \right)^{(1-a)/t}. 
\end{equation*}
Applying this inequality with $\sigma = p^*$ and $t = q$ and using 
Corollary~\ref{cor-S-P-1}, one obtains the conclusion. 
%If $p \ge d$, then applying this inequality with  $t = q$, and using Corollary~\ref{cor-S-P-2} and Jensen's inequality, one obtains the conclusion. 
\end{proof}

We also have, see \cite[Theorem on page 125 and the following remarks]{Nirenberg58}  
 
\begin{lemma}[Nirenberg's interpolation inequality] \label{lem-S-P-4} Let $d \ge 1$, $p \ge 1$, $q \ge 1$, $\rr > 0$, and $0 \le a \le 1$ be such that
$$
\frac{1}{\rr} \ge a \left(\frac{1}{p} - \frac{1}{d} \right) + \frac{1-a}{q}. 
$$
Let $0 < r < R$, and $\lambda > 0$  and set 
$$
\lambda D: =\big\{\lambda x \in \R^d: r  < |x| < R \big\}.
$$
Then, for $u \in L^1(\lambda D)$, 
\begin{equation*}
%\label{Poincare}
\left(\dashint_{\lambda D} \left| u -  \dashint_{\lambda D} u \right|^{\rr} \, dx  \right)^{1/\rr} \le C\| \nabla u\|_{L^p(\lambda D)}^a
C\| u \|_{L^q(\lambda D)}^{1-a},
\end{equation*}
for some positive constant $C$ independent of $u$, $\lambda$, and $\delta$. 
\end{lemma}

We prove the following  more general version of Theorem~\ref{thm-CKN}:

\begin{theorem}\label{CKN-g} Let $p \ge 1$, $q \ge 1$, $\rr>0$, $0 <   a \le   1$, $\alpha, \, \beta, \, \gamma \in \R$ be such that 
%\begin{equation}\label{CKN-1}
%\frac{1}{\rr} + \frac{\gamma}{d} \ge 0,  
%\end{equation}
\begin{equation}\label{CKN-balance}
\frac{1}{\rr} + \frac{\gamma}{d} = a \Big(\frac{1}{p} + \frac{\alpha - 1}{d} \Big) + (1-a) \Big( \frac{1}{q} + \frac{\beta}{d} \Big), 
\end{equation}
and, with $\gamma = a \sigma + (1 -a) \beta$, 
\begin{equation*}
%\label{CKN-sign}
0 \le  \alpha - \sigma \le 1.   
\end{equation*}
Set, for $k \in \Z$,  
\begin{equation}
I_\delta(k, u): = \left\{\begin{array}{cl} I_\delta(u, \C_k \cup \C_{k+1}, \alpha) + 2^{k (\alpha p + d - p)}\delta^p & \mbox{ if } 1 < p < d, \\[6pt]
\| |x|^\alpha \nabla u\|_{L^p(\C_k \cup \C_{k+1})}^p & \mbox{ otherwise}. 
\end{array}\right. 
\end{equation}
We have, for $u \in L^p_{\loc}(\R^d)$ and $m, n \in \Z$ with $m < n$, 
\begin{enumerate}

\item[i)] if $1/ \rr + \gamma/ d > 0$ and  $\supp u \subset B_{2^n}$,  then 
\begin{equation*}
\left(\int_{\R^d \setminus B_{2^m}}  |x|^{\gamma \rr}|u|^{\rr} \, dx \right)^{1/\rr} \le C\left( \sum_{k = m-1}^n I_\delta (k, u)  \right)^{a/p} \| |x|^\beta u   \|_{L^q(\R^d)}^{(1-a)}, 
\end{equation*}

\item[ii)] if $1/ \rr + \gamma/ d <  0$ and  $\supp u \subset \R^d \setminus B_{2^m}$,  then 
\begin{equation*}
\left(\int_{B_{2^n}}   |x|^{\gamma \rr} |u|^{\rr} \, dx \right)^{1/\rr} \le C\left(  \sum_{k = m-1}^n I_\delta (k, u) \right)^{a/p} \| |x|^\beta u   \|_{L^q(\R^d)}^{(1-a)}, 
\end{equation*}

\item[iii)] if $1/ \rr + \gamma/ d = 0$, $\rr > 1$,  and  $\supp u \subset B_{2^n}$,  then 
\begin{equation*}
\left(\int_{\R^d \setminus B_{2^m}}  \frac{ |x|^{\gamma \rr}}{\ln^{\rr}(2^{n+1}/ |x|)} |u|^{\rr} \, dx \right)^{1/\rr} \le C\left( \sum_{k = m-1}^n I_\delta (k, u)  \right)^{a/p} \| |x|^\beta u   \|_{L^q(\R^d)}^{(1-a)}, 
\end{equation*}

\item[iv)] if $1/ \rr + \gamma/ d = 0$, $\rr > 1$,  and  $\supp u \subset \R^d \setminus B_{2^m}$,  then 
\begin{equation*}
\left(\int_{B_{2^n}}  \frac{ |x|^{\gamma \rr}}{\ln^\rr(2^{n+1}/ |x|)} |u|^{\rr} \, dx \right)^{1/\rr} \le C\left(  \sum_{k = m-1}^n I_\delta (k, u)\right)^{a/p} \| |x|^\beta u   \|_{L^q(\R^d)}^{(1-a)}. 
\end{equation*}

\end{enumerate}
Here $C$ denotes a positive constant independent of $u$, $\delta$, $k$, $n$, and $m$.
\end{theorem}

\begin{proof} We only present the proof in the case $1 < p < d$. The proof for the other case follows similarly, however instead of using Lemma~\ref{lem-S-P-3}, one applies Lemma~\ref{lem-S-P-4}.  We now assume that $1 < p < d$. Since $\alpha - \sigma \ge 0$,  by Lemma~\ref{lem-S-P-3}, we have
\begin{equation}\label{CKN-claim}
\left(\dashint_{\C_k} \left| u - \dashint_{\C_k} u  \right|^\rr \, dx \right)^{1/\rr} \le C \Big(2^{-(d-p) k} I_\delta(u, \C_k) + \delta^p \Big)^{a/p}  \left( \dashint_{\C_k}|u|^q\right)^{(1-a)/q} . 
\end{equation}
Using \eqref{CKN-balance}, we derive from \eqref{CKN-claim} that
\begin{equation}\label{CKN-part1}
\int_{\C_k} |x|^{\gamma \rr}|u|^\rr \, dx  \le C 2^{(\gamma \rr + d) k }\Big|\dashint_{\C_k} u \Big|^\rr + C  \Big( I_\delta(u, \C_k, \alpha) + 2^{k (\alpha p + d - p)}\delta^p \Big)^{a \rr /p} \| |x|^\beta u   \|_{L^q(\C_k)}^{(1-a)\rr}. 
\end{equation}

\noindent
$\bullet$ {\bf Step 1}: Proof of $i)$. 
Summing \eqref{CKN-part1} with respect to $k$ from $m$ to $n$, we obtain 
\begin{align}\label{CKN-part2}
\int_{\{ |x| > 2^{m} \}}  |x|^{\gamma \rr}|u|^\rr \, dx  \le & C \sum_{k = m}^n 2^{(\gamma \rr + d)k} \left|\dashint_{\C_k} u\right|^{\rr}  \\[6pt]   
+ & C \sum_{k = m}^n  \Big( I_\delta(u, \C_k, \alpha) + 2^{k (\alpha p + d - p)}\delta^p \Big)^{a \rr /p} \| |x|^\beta u   \|_{L^q(\C_k)}^{(1-a) \rr}.  \nonumber
\end{align}
By Lemma~\ref{lem-S-P-3}, we have
\begin{equation*}
\left| \dashint_{\C_{k}} u \right| \le \left| \dashint_{\C_{k+1}} u \right| +  C \Big(2^{-(d-p) k} I_\delta(u, \C_k \cup \C_{k+1}) + \delta^p \Big)^{a/p}  \left( \dashint_{\C_k \cup \C_{k+1}}|u|^q\right)^{\frac{1-a}{q}}. 
\end{equation*}
Applying Lemma~\ref{lem-Holder}, we derive that 
\begin{equation*}
\left| \dashint_{\C_{k}} u \right|^\rr \le \frac{2^{\gamma \rr + d + 1}}{1 + 2^{\gamma \rr + d}} \left|\dashint_{\C_{k+1}} u \right|^\rr +  C \Big(2^{-(d-p) k} I_\delta(u, \C_k \cup \C_{k+1}) + \delta^p \Big)^{a \rr /p}  \left( \dashint_{\C_k \cup \C_{k+1}}|u|^q\right)^{\frac{(1-a)\rr}{q}}. 
\end{equation*}
It follows that, with  $c = 2 / (1 + 2^{\gamma \rr + d}) < 1$, 
\begin{align*}
2^{(\gamma \rr + d)k}\left| \dashint_{\C_{k}} u \right|^\rr \le &   c 2^{(\gamma \rr + d)(k+1)}\left| \dashint_{\C_{k+1}} u \right|^\rr  \\[6pt]
+ &  C \Big( I_\delta(u, \C_k \cup \C_{k+1}, \alpha) + 2^{k (\alpha p + d - p)}\delta^p \Big)^{a \rr/p} \| |x|^\beta u   \|_{L^q(\C_k \cup \C_{k+1})}^{(1-a)\rr}. 
\end{align*}
This yields
\begin{equation}\label{CKN-part3}
\sum_{k = m}^n 2^{(\gamma \rr + d)k}\left| \dashint_{\C_{k}} u \right|^\rr  \le C \sum_{k = m}^{n}  \Big( I_\delta(u, \C_k \cup \C_{k+1}, \alpha) + 2^{k (\alpha p + d - p)}\delta^p \Big)^{a \rr /p} \| |x|^\beta u   \|_{L^q(\C_k \cup \C_{k+1})}^{(1-a)\rr}. 
\end{equation}
Combining \eqref{CKN-part2} and \eqref{CKN-part3} yields 
\begin{multline}\label{CKN-part3-*}
\int_{\{ |x| > 2^{m} \}}  |x|^{\gamma \rr}|u|^\rr \, dx  \\[6pt] 
\le C \sum_{k = m-1}^n  \Big( I_\delta(u, \C_k \cup \C_{k+1}, \alpha) + 2^{k (\alpha p + d - p)}\delta^p \Big)^{a \rr /p} \| |x|^\beta u   \|_{L^q(\C_k \cup \C_{k+1})}^{(1-a) \rr}.
\end{multline}
Applying the inequality, for $s \ge 0$, $t \ge 0$ with $s + t \ge 1$, and for $x_k \ge 0$ and $y_k \ge 0$,  
$$
\sum_{k = m}^n x_k^s y_k^t \le C_{s, t} \Big(\sum_{k = m}^n x_k \Big)^s \Big(\sum_{k = m}^n y_k \Big)^t,  
$$
to $s =  a \rr /p$ and $t = (1-a) \rr / q$,  we obtain from \eqref{CKN-part3-*} that
\begin{equation}\label{CKN-part4}
\int_{\{ |x|  > 2^{m} \}}  |x|^{\gamma \rr}|u|^\rr \, dx  \le C  \left( \sum_{k = m}^n I_{\delta}(k, u) \right)^{a \rr /p} \| |x|^\beta u   \|_{L^q(\R^d)}^{(1-a)\rr}
\end{equation}
since  $
a / p + (1 -a)/q  \ge 1/ \rr \mbox{ thanks to the fact } \alpha - \sigma - 1 \le 0. 
$

\medskip 
\noindent
$\bullet$ {\bf Step 2}: Proof of $ii)$. The proof is in the spirit of the proof of $ii)$ of Theorem~\ref{thm-Hardy}. The details are left to the reader.

\medskip 
\noindent
$\bullet$ {\bf Step 3:} Proof of $iii)$. Fix $\xi > 0$. 
Summing \eqref{CKN-part1} with respect to $k$ from $m$ to $n$, we obtain 
\begin{multline}\label{CKN-part1-2}
\int_{\{ |x| > 2^{m} \}} \frac{1}{\ln^{1 + \xi} (\rr/ |x|)} |x|^{\gamma \rr}|u|^\rr \, dx  \\[6pt]
\le C \sum_{k = m}^n \frac{1}{(n-k+1)^{1 + \xi}} \left|\dashint_{\C_k} u\right|^\rr   + C \sum_{k = m}^n  \Big( I_\delta(u, \C_k, \alpha) + 2^{k (\alpha p + d - p)}\delta^p \Big)^{a \rr /p} \| |x|^\beta u   \|_{L^q(\C_k)}^{(1-a)\rr}. 
\end{multline}
By Lemma~\ref{lem-S-P-3}, we have
\begin{equation*}
\left| \dashint_{\C_{k}} u \right|  \le  \left|\dashint_{\C_{k+1}} u \right| + C \Big(2^{-(d-p) k} I_\delta(u, \C_k \cup \C_{k+1}) + \delta^p \Big)^{a/p}  \left( \dashint_{\C_k \cup \C_{k+1}}|u|^q\right)^{\frac{1-a}{q}}. 
\end{equation*}
Applying Lemma~\ref{lem-Holder} with 
$$
c = \frac{(n-k+1)^\xi}{ (n-k+1/2)^\xi},
$$
 we deduce that
\begin{multline}\label{CKN-part2-2}
\frac{1}{(n - k + 1)^{\xi}} \left| \dashint_{\C_{k}} u \right|^{\rr} \le \frac{1}{(n - k + 1/2)^{\xi}}  \left| \dashint_{\C_{k+1}} u \right|^{\rr} \\[6pt]+  C (n - k + 1)^{\rr-1 - \xi} \Big(2^{-(d-p) k} I_\delta(u, \C_k \cup \C_{k+1}) + \delta^p \Big)^{a \rr/p}  \left( \dashint_{\C_k \cup \C_{k+1}}|u|^q\right)^{\frac{(1-a) \rr}{q}}. 
\end{multline}
Recall that, for $k \le n$ and $\xi >0$,  
\begin{equation}\label{CKN-part3-2}
\frac{1}{(n - k + 1)^{\xi}} - \frac{1}{(n - k + 3/2)^{\xi}} \sim \frac{1}{(n - k + 1)^{\xi + 1}}. 
\end{equation}
Taking $\xi = \rr -1 $, we derive from \eqref{CKN-part2-2} and \eqref{CKN-part3-2} that
\begin{equation}\label{CKN-part4-2}
\sum_{k = m}^n 2^{(\gamma \rr + d)k} \frac{1}{(n-k+1)^{\rr}}  \left| \dashint_{\C_{k}} u \right|^\rr \le \sum_{k = m}^n  C \Big( I_\delta(k, u) \Big)^{a \rr/p} \| |x|^\beta u   \|_{L^q(\C_k \cup \C_{k+1})}^{(1-a)\rr}. 
\end{equation}
%We derive that 
%\begin{equation}\label{CKN-part4-2}
%\sum_{k = m}^n 2^{(\gamma \rr + d)k}\Big| \dashint_{\C_{k}} u \Big|^\rr  \le C \sum_{k = m}^{n}  \Big( J_\delta(u, \C_k \cup C_{k+1}) + 2^{k (\alpha p + d - p)}\delta^p \Big)^{a \rr /p} \| |x|^\beta u   \|_{L^q(\C_k \cup C_{k+1})}^{(1-a)\rr}. 
%\end{equation}
Combining \eqref{CKN-part1-2} and \eqref{CKN-part4-2}, as in \eqref{CKN-part4}, we obtain  
\begin{equation*}
\int_{\{ |x|  > 2^{m} \}} \frac{|x|^{\gamma \rr}}{\ln^{\rr} (2^{n+1}/|x|)}  |u|^\rr \, dx  \le C  \left( \sum_{k = m}^n I_\delta(k, u)  \right)^{a \rr /p} \| |x|^\beta u   \|_{L^q(\R^d)}^{(1-a)\rr}.
\end{equation*}

\medskip
\noindent
$\bullet$ {\bf Step 4}: Proof of $iv)$. The proof is in the spirit of the proof of $iv)$ of Theorem~\ref{thm-Hardy}. The details are left to the reader.

\medskip
The proof is complete. 
\end{proof}

\begin{remark} \label{rem-CKN-g} \rm For $p> 1$, we have (see \cite[Theorem 4]{nguyen06})
$$
I_\delta(k, u) \le C \int_{\C_k \cup \C_{k+1}} |x|^{p \alpha} |\nabla u |^p \, dx \mbox{ for } k \in \Z, 
$$
for some positive constant $C$ independent of $k$ and $u$. This implies 
$$
\left( \sum_{k = m-1}^n I_\delta(k, u) \right)^{1/p} \le C \| |x|^\alpha \nabla u \|_{L^p(\R^d)}. 
$$
From Theorem~\ref{CKN-g}, one then obtains improvement of Caffarelli-Kohn-Nirenberg's  inequality for the case $0 \le \alpha - \sigma \le 1$ and for $1 < p < d$. 
\end{remark}

Using Theorem~\ref{CKN-g}, we can derive that 

\begin{proposition}\label{CKN-g-1} Let $p \ge 1$, $q \ge 1$, $\rr>0$, $0 <   a  <  1$, $\alpha, \, \beta, \, \gamma \in \R$ be such that 
%\begin{equation}\label{CKN-1}
%\frac{1}{\rr} + \frac{\gamma}{d} \ge 0,  
%\end{equation}
\begin{equation*}
\frac{1}{\rr} + \frac{\gamma}{d} = a \Big(\frac{1}{p} + \frac{\alpha - 1}{d} \Big) + (1-a) \Big( \frac{1}{q} + \frac{\beta}{d} \Big), 
\end{equation*}
and, with $\gamma = a \sigma + (1 -a) \beta$, 
\begin{equation*}
%\label{CKN-sign}
 \alpha - \sigma > 1   \quad \mbox{ and } \quad %\label{CKN-a-3}
\frac{1}{\rr} + \frac{\gamma}{d}  \neq  \frac{1}{p} + \frac{\alpha - 1}{d}. 
\end{equation*}
We have, for $u \in C^1_c(\R^d)$, 
\begin{enumerate}

\item[i)] if $1/ \rr + \gamma/ d > 0$,  then 
\begin{equation*}
\Big(\int_{\R^d}  |x|^{\gamma \rr}|u|^{\rr} \, dx \Big)^{1/\rr} \le C \| |x|^\alpha \nabla u \|_{L^p(\R^d)}^{a} \| |x|^\beta u   \|_{L^q(\R^d)}^{(1-a)},
\end{equation*}

\item[ii)] if $1/ \rr + \gamma/ d <  0$ and  $\supp u \subset \R^d \setminus \{0 \}$,  then 
\begin{equation*}
\Big(\int_{\R^d}  |x|^{\gamma \rr}|u|^{\rr} \, dx \Big)^{1/\rr} \le C \| |x|^\alpha \nabla u \|_{L^p(\R^d)}^{a} \| |x|^\beta u   \|_{L^q(\R^d)}^{(1-a)},
\end{equation*}
\end{enumerate}
for some positive constant $C$ independent of $u$. 
\end{proposition}

\begin{proof}
 The proof is in the spirit of the approach in \cite{CKN} (see also \cite{NgS3}). Since 
$$
\frac{1}{p} + \frac{\alpha -1 }{d} \neq \frac{1}{q} + \frac{\beta}{d}. 
$$
by scaling, one might assume that 
$$
 \||x|^\alpha \nabla u \|_{L^p(\R^d)}  = 1 \quad \mbox{ and } \quad \| |x|^\beta u \|_{L^q(\R^d)} = 1. 
$$
Let $0 < a_2 < 1$ be such that 
\begin{equation}\label{CKN-2-cond0}
|a_2 -a| \mbox{ is small enough},
\end{equation}
 and set 
\begin{equation*}
\frac{1}{\rr_2}  = \frac{a_2}{p}  + \frac{1-a_2}{q}  \quad \mbox{ and } \quad \gamma_2 = a_2(\alpha-1) + (1 - a_2) \beta. 
\end{equation*}
We have 
\begin{equation}\label{CKN-2-part2}
\frac{1}{\rr_2} + \frac{\gamma_2}{d} = a_2 \Big( \frac{1}{p} + \frac{\alpha-1}{d}\Big) + (1 - a_2) \Big( \frac{1}{q} + \frac{\beta}{d}\Big).
\end{equation}
Recall that 
\begin{equation}\label{CKN-2-part3}
\frac{1}{\rr} + \frac{\gamma}{d} = a \Big( \frac{1}{p} + \frac{\alpha-1}{d}\Big) + (1 - a) \Big( \frac{1}{q} + \frac{\beta}{d}\Big).
\end{equation}
Since $a>0$ and $\alpha - \sigma >1 $,  it follows from \eqref{CKN-2-cond0} that
\begin{equation}\label{CKN-2-toto-2}
\frac{1}{\rr} - \frac{1}{\rr_2} = (a -a_2) \Big( \frac{1}{p} - \frac{1}{q}  \Big) + \frac{a}{d} (\alpha - \sigma - 1) > 0.
\end{equation}

We first  choose $a_2$  such  that 
\begin{equation}\label{CKN-2-cond1}
a_2 < a  \quad \mbox{ if } \quad \frac{1}{p} + \frac{\alpha -1}{d}  < \frac{1}{q} + \frac{\beta}{d},
\end{equation}
\begin{equation}\label{CKN-2-cond2}
a < a_2 \quad \mbox{ if } \quad \frac{1}{p} + \frac{\alpha -1}{d}  > \frac{1}{q} + \frac{\beta}{d}.
\end{equation}
Using \eqref{CKN-2-cond0},  \eqref{CKN-2-cond1} and \eqref{CKN-2-cond2}, we derive from  \eqref{CKN-2-part2}, and \eqref{CKN-2-part3} that
\begin{equation}\label{CKN-2-choice2}
 \frac{1}{\rr} + \frac{\gamma}{d}  < \frac{1}{\rr_2} + \frac{\gamma_2}{d} \quad \mbox{ and } \quad \left( \frac{1}{\rr} + \frac{\gamma}{d}  \right) \left( \frac{1}{\rr_2} + \frac{\gamma_2}{d} \right) > 0.  
\end{equation}
It follows from \eqref{CKN-2-toto-2}, \eqref{CKN-2-choice2}, and H\"older's inequality that 
\begin{equation*}
\| |x|^\gamma u \|_{L^\rr(\R^d \setminus B_1)} \le C \| |x|^{\gamma_2} u \|_{L^{\rr_2}(\R^d)}.  
\end{equation*}
Applying Theorem~\ref{CKN-g} (see also Remark~\ref{rem-CKN-g}), we have 
\begin{equation*}
 \| |x|^{\gamma_2} u \|_{L^{\rr_2}(\R^d)} \le C  \||x|^\alpha \nabla u \|_{L^p(\R^d)}^{a_2} \| |x|^\beta u   \|_{L^q(\R^d)}^{(1-a_2)}  \le C, 
\end{equation*}
which yields 
\begin{equation}\label{coucoucou1}
\| |x|^\gamma u \|_{L^\rr(\R^d \setminus B_1)}  \le C. 
\end{equation}

We next  choose $a_2$ such  that 
\begin{equation}\label{CKN-2-cond1-1}
a < a_2  \quad \mbox{ if } \quad \frac{1}{p} + \frac{\alpha -1}{d}  < \frac{1}{q} + \frac{\beta}{d},
\end{equation}
\begin{equation}\label{CKN-2-cond2-1}
a_2 < a \quad \mbox{ if } \quad \frac{1}{p} + \frac{\alpha -1}{d}  > \frac{1}{q} + \frac{\beta}{d}.
\end{equation}
Using \eqref{CKN-2-cond0},  \eqref{CKN-2-cond1-1} and \eqref{CKN-2-cond2-1}, we derive from  \eqref{CKN-2-part2}, and \eqref{CKN-2-part3} that
\begin{equation}\label{CKN-2-choice2-1}
\frac{1}{\rr_2} + \frac{\gamma_2}{d}  < \frac{1}{\rr} + \frac{\gamma}{d}  \quad \mbox{ and } \quad \left( \frac{1}{\rr} + \frac{\gamma}{d}  \right) \left( \frac{1}{\rr_2} + \frac{\gamma_2}{d} \right) > 0.  
\end{equation}
It follows from \eqref{CKN-2-toto-2}, \eqref{CKN-2-choice2-1}, and H\"older's inequality that 
\begin{equation*}
\| |x|^\gamma u \|_{L^\rr(B_1)} \le C \| |x|^{\gamma_2} u \|_{L^{\rr_2}(\R^d)}.  
\end{equation*}
Applying Theorem~\ref{CKN-g} (see also Remark~\ref{rem-CKN-g}), we have 
\begin{equation*}
 \| |x|^{\gamma_2} u \|_{L^{\rr_2}(\R^d)} \le C  \||x|^\alpha \nabla u \|_{L^p(\R^d)}^{a_2} \| |x|^\beta u   \|_{L^q(\R^d)}^{(1-a_2)}  \le C, 
\end{equation*}
which yields 
\begin{equation}\label{coucoucou2}
\| |x|^\gamma u \|_{L^\rr(\R^d \setminus B_1)}  \le C. 
\end{equation}
The conclusion now follows from \eqref{coucoucou1} and \eqref{coucoucou2}. 
\end{proof}

\begin{remark} \label{rem-Maximal} \rm 

Using the approach in the proof of \cite[Theorem 2]{nguyen06}, one can prove that, for $p>1$,  
\begin{equation}\label{rem-claim0}
I_\delta(u, \alpha) \le C \int_{\R^d} \int_{\mS^{d-1}}|x|^{p \alpha} 
|{\mathcal M}(\sigma, \nabla u) (x)|^p \, d \sigma \, dx, 
\end{equation}
where 
$$
{\mathcal M}(\sigma, \nabla u) (x): = 
\sup_{r > 0} \frac{1}{r}\int_0^r | \nabla u (x + s \sigma) \cdot \sigma| \, ds. 
$$
We claim that, for $-1/ p < \alpha < 1 - 1/ p$, it holds
\begin{equation}\label{rem-claim}
\int_{\R^d}|x|^{p \alpha} |{\mathcal M}(\sigma, \nabla u) (x)|^p \, d \sigma \, dx \le C \int_{\R^d}|x|^{p \alpha} |\nabla u (x) \cdot \sigma|^p \, dx,\quad 
\mbox{ for all } \sigma \in \mS^{d-1}, 
\end{equation}
for some positive constant $C$ independent of $\sigma$ and $u$. Then, combining \eqref{rem-claim0} and \eqref{rem-claim} yields 
\begin{equation}
I_\delta (u, \alpha) \le C \int_{\R^d} |x|^{p \alpha} |\nabla u|^p \, dx. 
\end{equation}
as mentioned in Remark~\ref{rem-Muck}.
For simplicity, we assume that 
$\sigma = e_d = (0, \cdots, 0, 1) \in \R^d$ and prove \eqref{rem-claim}.  We have, for any bounded interval $(a, b)$ and for any $x' \in \R^{d-1}$ 
\begin{equation}
\label{weightass}
\dashint_{a}^b (|x'| + |s|)^{p \alpha}  \, ds \left( \dashint_{a}^b (|x'| + |s|)^{- p \alpha/ (p-1)} \, ds  \right)^{p-1} \le C, 
\end{equation}
for some positive constant $C$ independent of $(a, b)$ and $x'$ since $-1/ p < \alpha < 1 - 1/ p$.  Applying the theory of maximal functions with weights due to 
Muckenhoupt \cite[Corollary 4]{Muck} (see also \cite[Theorem 1]{CF}), 
which holds whenever the weight satisfies \eqref{weightass}, we obtain 
\begin{align*}
\int_{\R^d} |x |^{p \alpha} |{\mathcal M}(e_d, \nabla u) (x )|^p \, dx &\le  C  \int_{\R^{d-1}}\int_{\R} (|x' | + |x_d|)^{p \alpha} |{\mathcal M}(e_d, \nabla u) (x', x_d )|^p \, dx_d \, d x' \\[6pt]
& \le  C \int_{\R^{d-1}}\int_{\R} (|x' | + |x_d|)^{p \alpha} |\partial_{x_d} u (x', x_d ) |^p \, dx_d \, d x' \\
& \le C \int_{\R^d} |x|^{p \alpha} |\nabla u |^p \, dx. 
\end{align*}
The claim \eqref{rem-claim} is proved.

\end{remark}

\section{Results in bounded domains}
\label{boundedcase}

In this section, we present some results in the spirit of Theorems~\ref{thm-Hardy} and \ref{CKN-g} for a smooth bounded domain $\Omega$.  As a consequence of Theorem~\ref{thm-Hardy} and the extension argument in the proof of Lemma~\ref{lem-S}, we obtain 

\begin{proposition}\label{pro-Hardy-domain} Let $d \ge 1$, $1 \le p \le d$, $\Omega \Subset B_R$ a smooth open subset of $\R^d$, and $u \in L^p(\Omega)$. We have
\begin{enumerate}

\item[i)]  if $1 \le p < d$, then 
\begin{equation*}
\int_{\Omega} \frac{|u(x)|^p}{|x|^p} \, dx \le C_\Omega \left(I_\delta(u, \Omega) + \| u \|_{L^p(\Omega)}^p  + \delta^p \right), 
\end{equation*}

\item[ii)] if $p > d$ and $\supp u \subset \bar \Omega \setminus B_r$, then
\begin{equation*}
\int_{\Omega} \frac{|u(x)|^p}{|x|^p} \, dx \le C_\Omega \left(I_\delta(u, \Omega)  + \| u \|_{L^p(\Omega)}^p + r^{d-p} \delta^p  \right), 
\end{equation*}

\item[iii)] if $p = d \ge 2$, then
\begin{equation*}
\int_{\Omega \setminus B_r} \frac{|u(x)|^d}{|x|^d \ln^d (2R/ |x|)} \, dx \le C_\Omega \left(I_\delta(u, \Omega) + \| u \|_{L^p(\Omega)}^p +   \ln (2R/ r) \delta^d  \right), 
\end{equation*}
%and, for $\beta > 3$,  
%\begin{equation*}
%\int_{\R^d} \frac{|u(x)|^d}{|x|^d \ln^{\beta} (2R/ |x|)} \, dx \le C_\beta \left(I_\delta(u) +   \delta^d  \right).   
%\end{equation*}

\item[iv)] if $p = d \ge 2$ and $\supp u \subset \Omega \setminus B_r$, then
\begin{equation*}
\int_{\Omega \cap B_R} \frac{|u(x)|^d}{|x|^d \ln^d (2|x|/r)} \, dx \le C_\Omega \left(I_\delta(u, \Omega) + \| u \|_{L^p(\Omega)}^p +   \ln (2R/ r) \delta^d  \right),  
\end{equation*}
 
\end{enumerate}
Here  $C_\Omega$ denotes a positive constant depending only on $p$ and $\Omega$. 

\end{proposition}

Using Theorem~\ref{thm-CKN}, we derive 

\begin{proposition}\label{pro-CKN-domain} Let $d \ge 2$, $1 < p < d$, $q \ge 1$, $\rr>0$, $0 <   a \le   1$, $\alpha, \, \beta, \, \gamma \in \R$, $0 \in \Omega  \subset B_R$ a smooth bounded open subset of $\R^d$, and $u \in L^p(\Omega)$ be such that
\begin{equation*}
%\label{CKN-balance-O}
\frac{1}{\rr} + \frac{\gamma}{d} = a \Big(\frac{1}{p} + \frac{\alpha - 1}{d} \Big) + (1-a) \Big( \frac{1}{q} + \frac{\beta}{d} \Big), 
\end{equation*}
and, with $\gamma = a \sigma + (1 -a) \beta$, 
\begin{equation*}
%\label{CKN-sign}
0 \le  \alpha - \sigma \le 1.   
\end{equation*}
We have
\begin{enumerate}

\item[i)] if $1/ \rr + \gamma/ d > 0$,  then 
\begin{equation*}
\left(\int_{\Omega}  |x|^{\gamma \rr}|u|^{\rr} \, dx \right)^{1/\rr} \le C\Big(  I_\delta (u, \Omega, \alpha)  + \| u\|_{L^p(\Omega)}^p + \delta^p \Big)^{a/p} \| |x|^\beta u   \|_{L^q(\Omega)}^{(1-a)}, 
\end{equation*}

\item[ii)] if $1/ \rr + \gamma/ d <  0$ and  $\supp u \subset \Omega \setminus \{ 0 \} $,  then 
\begin{equation*}
\left(\int_{\Omega}   |x|^{\gamma \rr} |u|^{\rr} \, dx \right)^{1/\rr} \le C\Big(   I_\delta (u, \Omega, \alpha)  + \| u\|_{L^p(\Omega)}^p   + \delta^p\Big)^{a/p} \| |x|^\beta u   \|_{L^q(\Omega)}^{(1-a)}, 
\end{equation*}

\item[iii)] if $1/ \rr + \gamma/ d = 0$ and $\rr > 1$,  then 
\begin{equation*}
\left(\int_{\Omega \setminus B_r}  \frac{ |x|^{\gamma \rr}}{\ln^{\rr}(2R/ |x|)} |u|^{\rr} \, dx \Big)^{1/\rr} \le C\Big( I_\delta (u, \Omega, \alpha)  + \| u\|_{L^p(\Omega)}^p   + \delta^p \ln (2 R/ r) \right)^{a/p} \| |x|^\beta u   \|_{L^q(\Omega)}^{(1-a)}, 
\end{equation*}

\item[iv)] if $1/ \rr + \gamma/ d = 0$, $\rr > 1$,  and $\supp u \subset \Omega \setminus B_r$,  then 
\begin{equation*}
\left(\int_{\Omega}  \frac{ |x|^{\gamma \rr}}{\ln^\rr(2 |x|/ r)} |u|^{\rr} \, dx \right)^{1/\rr} \le C\left( I_\delta (u, \Omega, \alpha)  + \| u\|_{L^p(\Omega)}^p   + \delta^p \ln (2 R/ r) \right)^{a/p} \| |x|^\beta u   \|_{L^q(\Omega)}^{(1-a)}.
\end{equation*}

\end{enumerate}
Here  $C$  denotes a positive constant  independent of $u$ and $\delta$.  

\end{proposition}
\begin{proof} Let $v$ be the extension of $u$ in $\R^d$ as in the proof of Lemma~\ref{lem-S}. As in the proof of Lemma~\ref{lem-S},  we have, since $0 \in \Omega$, 
	$$
	I_{2\delta}(v, \alpha) \le C \Big( I_\delta(u, \Omega, \alpha) + \| u\|_{L^p(\Omega)} \Big). 
	$$
	We also have, since $0 \in \Omega$,  
	$$
	\| |x|^\beta v   \|_{L^q(\R^d)} \le C \| |x|^\beta u   \|_{L^q(\Omega)}.
	$$
	The conclusion now follows from Theorem~\ref{CKN-g}. 
\end{proof}

\bigskip

\end{document}